\documentclass{amsart}
\usepackage[usenames,dvipsnames]{color}

\newtheorem{thm}{Theorem}[section]
\newtheorem{defn}[thm]{Definition}
\newtheorem{cor}[thm]{Corollary}
\newtheorem{lem}[thm]{Lemma}
\newtheorem{prop}[thm]{Proposition}
\newtheorem{rem}[thm]{Remark}
 \numberwithin{equation}{section}
 
 \DeclareMathOperator{\sgn}{sgn}

\begin{document}

\title{$(E,F)$-multipliers and applications}
\author{F.~Sukochev$^{*}$}
\address{F.S.: School of Mathematics and Statistics, University of New South Wales, Sydney NSW 2052, Australia}
\email{f.sukochev@unsw.edu.au}
\author{A.~Tomskova$^{*}$}
\address{A.T.: Institute of Mathematics and IT, Durmon yuli 29, Tashkent,
Uzbekistan}
\email{tomskovaanna@gmail.com}
\thanks{$^{*}$ Research was partially supported by the ARC}
\date{}

\begin{abstract} For two given symmetric sequence spaces $E$ and
$F$ we study the $(E,F)$-multiplier space, that is the space all
of matrices $M$ for which the Schur product $M\ast A$ maps $E$
into $F$ boundedly whenever $A$ does. We obtain several results
asserting continuous embedding of  $(E,F)$-multiplier space into
the classical $(p,q)$-multiplier space (that is when $E=l_p$,
$F=l_q$). Furthermore, we present many examples of symmetric
sequence spaces $E$ and $F$ whose projective and injective tensor
products are not isomorphic to any subspace of a Banach space with
an unconditional basis, extending classical results of S.
Kwapie\'{n} and A. Pe{\l}czy\'{n}ski \cite{Pil} and of G. Bennett
\cite{Ben} for the case when $E=l_p$, $F=l_q$.
\end{abstract}
\maketitle
\section{Introduction}
  For an infinite scalar-valued (real or complex) matrix
$A=(a_{ij})_{i,j=1}^\infty$ and $n\in \mathbb N$ we set
$$T_n(A):=(t^{(n)}_{ij})_{i,j=1}^\infty,\quad\text {where}\quad t^{(n)}_{ij}=
a_{ij},\quad\text{ for} \quad 1\leq j\leq i\leq n$$ and
 $t^{(n)}_{ij}=0$ otherwise. The operator $T_n$ is called the
$n$-th main triangle projection. S. Kwapie\'{n} and A.
Pe{\l}czy\'{n}ski
 studied in \cite{Pil} the norms of the operators $(T_n)_{n\geq 1}$
 acting on the space of all bounded linear operators $B(l_p,l_q)$ and obtained that
 $\sup\limits_n\|T_n\|_{B(l_p,l_q)\to B(l_p,l_q)}=\infty$  for
$1\leq q\leq
 p\leq\infty$, $q\neq\infty$, $p\neq 1$. Moreover, as an application, they have established that
for $1<p\leq\infty$, $1<q\leq\infty$ and
$\frac{1}{p}+\frac{1}{q}\leq 1$ (respectively, $1\leq p<\infty$,
$1\leq q<\infty$ and $\frac{1}{p}+\frac{1}{q}\geq 1$)
  projective (respectively, injective) tensor product of the spaces $l_p$ and $l_q$ is not isomorphic to any
subspace of a Banach space with an unconditional basis. In the
same paper \cite{Pil} the question (Problem 1) whether the
sequence $(\|T_n\|_{B(l_p,l_q)\to B(l_p,l_q)})_{n\geq 1}$ is
bounded for $1< p<
 q<\infty$ was stated. The positive answer to that question was obtained by
G. Bennett in his article \cite{Ben1}, where he established that
the main triangle projection $T$ defined on an element
$A=(a_{ij})_{i,j=1}^\infty\in B(l_p,l_q)$ by
$T(A):=(t_{ij})_{i,j=1}^\infty$ where $t_{ij}= a_{ij}$ for $1\leq
j\leq i$ and  $t_{ij}=0$ otherwise is bounded for $1< p<
 q<\infty$. G. Bennett obtained his result on the operator $T$ in the
 framework of the general theory of Schur multipliers on
 $B(l_p,l_q)$ (briefly, $(p,q)$-multipliers). For a deep study and applications of this notion in
 analysis and operator theory we refer to \cite{Ben1, Ben, Pis}.

The classical Banach spaces $l_p$, $(1\leq p\leq\infty)$ is an
 important representative of the class of symmetric sequence
 spaces (see e.g. \cite{Lin1}). The present paper extends results from
 \cite{Pil} and \cite{Ben} to a wider class of symmetric sequence spaces satisfying
 certain convexity conditions. In particular, we present sufficient conditions in terms of $p$-convexity and $q$-concavity of
symmetric sequence spaces guaranteeing that their projective and injective tensor
products are not isomorphic to any subspace of a Banach space with an
unconditional basis (Section 6, Theorem \ref{th9}). Our methods
are based on the study of general Schur multipliers on the space
 $B(E,F)$ (briefly $(E,F)$-multipliers) extending and generalizing
 several results from \cite{Ben}. In particular, we establish a number of
 results concerning the embedding of an  $(E,F)$-multiplier space into a
$(p,q)$-multiplier space and their coincidence (Section 4, Theorems \ref{th5} and
\ref{th6}).

An important technical tool used in this paper  is the theory of
generalized K\"{o}the duality (Section 3), which (to our best
knowledge) was firstly introduced by Hoffman \cite{Hoff} and
presented in a detailed manner in \cite{MP} (see also recent
papers \cite{CDS} and \cite{DS}).

In the final section (Section 6), we present an extension of
Kwapie\'{n} and Pe{\l}czy\'{n}ski  results for $l_p$-spaces to a wide class of Orlicz-Lorentz sequence spaces (Theorem \ref{th3}).\\

\textbf{Acknowledgements.} The authors are grateful to V. Chilin
and D. Potapov for their comments on earlier versions of this
paper. We also thank  A. Kaminska for her interest and additional references to the text.

\section{\textbf{Preliminaries and notation}}
Let $c_{0}$ be a linear space of all converging to zero real
sequences. For every $x\in c_0$,  by $|x|$ we denote the sequence
$(|x_i|)_{i=1}^\infty$ and by $x^*$ the non-increasing rearrangement of $|x|$, that is $x^*=(x_i^*)_{i=1}^\infty\in c_0$, where
$$
x_i^*=|x_{n_i}| \quad (i=1,2,...),
$$
where $(n_i)_{i=1}^\infty$ is a such permutation of natural
numbers, that the sequence $(|x_{n_i}|)$ is non-increasing.

In this paper, we work with symmetric sequence spaces which are a ~\lq close relative\rq~ of the classical $l_p$-spaces, $1\leq p\leq \infty$ (see \cite{Lin1, Lin2}).

Recall that a linear space $E\subset c_{0}$ equipped with a Banach norm
$\|\cdot\|$ is said to be a \textit{symmetric sequence space}, if
the following conditions hold:

(i) if $x,y\in E$ and $|x|\leq|y|$ , then $\|x\|\leq \|y\|$;

(ii) if $x\in E$, then $x^*\in E$ and $\|x^*\|= \|x\|$.

Without loss of generality we shall assume that
$\|(1,0,0,...)\|=1$.

A symmetric sequence space $E$ is said to be {\it
$p$--convex} ($1 \leq p \le\infty$), respectively, {\it
$q$--concave} ($1 \leq q \le \infty$) if
$$
\|(\sum_{k=1}^n |x_k|^p)^{1/p}\| \leq C (\sum_{k=1}^n
\|x_k\|_E^p)^{1/p},
$$
respectively,
$$
(\sum_{k=1}^n\|x_k\|_E^q)^{1/q} \leq C \| (\sum_{k=1}^n
|x_k|^q)^{1/q}\|
$$
(with a natural modification in the case $p=\infty$ or $q=\infty$)
for some constant $C>0$ and every choice of vectors $x_{1}, x_{2},
\dots, x_{n}$ in $X$. The least such constant is denoted by
$M^{(p)}(E)$ (respectively, $M_{(q)}(E)$)(see e.g. \cite{Lin2}).

\begin{rem}\label{r4} Any symmetric sequence space is 1-convex and
$\infty$-concave with constants equal to 1.
\end{rem}

The following proposition links $p$-convex and $q$-concave
sequence spaces to classical  $l_p$-spaces.

\begin{prop}\label{p154}\cite[p. 132]{Lin2}
If a symmetric sequence space $E$ is $p$-convex and $q$-concave,
then \begin{equation}\label{emb}l_p\subset E\subset
l_q\end{equation} and
\begin{equation}\label{emb}\|\cdot\|_q/M_{(q)}(E)\leq
\|\cdot\|\leq M^{(p)}(E)\|\cdot\|_p.\end{equation}
\end{prop}

Without loss of generality we shall assume that the embedding
constants in \eqref{emb} are both equal to 1 \cite[Proposition 1.d.8]{Lin2}.


Below, we restate the result given in \cite[Proposition
1.d.4 (iii)]{Lin2} for the case of symmetric
sequence spaces.(Here, by $E^*$ we denote the Banach dual of
$E$.)

\begin{prop}\label{p2}
Let $1\leq p,q \leq\infty$ be such that
$\frac{1}{p}+\frac{1}{q}=1$. A separable symmetric space $E$ is
$p$-convex (concave) if and only if the space $E^*$ is $q$-concave
(convex).
\end{prop}

\begin{rem}\label{r6} Let $E$ be $q$-concave.
If $E$ is not separable, then $E$ does not have order-continuous
norm. It follows from \cite[Chapter 10, \S4]{KA} that there exists
a pairwise disjoint sequence $(z_n)_n\subset E$ such that $z_n\geq
0$ and $\|z_n\|_E=1$, $n=1, 2 ,...$ and $x =
\sum\limits_{n=1}^\infty z_n\in E$, and this contradicts to
$q$-concavity of $E$. So if $E$ is $q$-concave, then $E$ is
separable and then $E^*$ is a symmetric sequence space.
\end{rem}

Given a symmetric sequence space $E$ and $0<p<\infty$ we denote
$$E^p:=\{x\in c_0: |x|^p\in E\},\quad
\|x\|_{E^p}:=(\|\ |x|^p\|_{E})^{1/p}.$$ The space $(E^p,
\|x\|_{E^p})$ is called the p-convexification of $E$ if $p>1$ and
the p-concavification of $E$ if $p<1$ (see e.g. \cite[Chapter 1.d]{Lin2}).

\begin{rem}\label{r2}
(i) If $1\leq p< \infty$, then $(E^p, \|\cdot\|_{E^p})$ is a
Banach space (for the proof see \cite[Proposition 1]{MP}). It is
also clear from the definition that $(E^p, \|\cdot\|_{E^p})$ is a
symmetric sequence space.

(ii) Generally speaking, the space $((E)^{1/p},
\|\cdot\|_{E^{1/p}})$ is not a Banach space, but if $E$ is
$p$-convex, then $(E)^{1/p}$ is a Banach space and so it is a
symmetric sequence space (for details see \cite[Chapter 1, p.
54]{Lin2}). Furthermore, this is not difficult to check that if
$E$ is $p$-convex, then the space $((E)^{1/p})^p$ is isometrically
isomorphic to $E$.
\end{rem}

\section{Generalized K\"{o}the duality}

For a symmetric sequence space $E$ by $E^\times$ we denote its
\textit{{K\"{o}the dual}}, that is
$$
E^\times:=\{y\in l_\infty:\sum\limits_{n=1}^\infty|x_ny_n|<\infty
\;\text{for every} \; x\in E\},
$$
and for $y\in E^\times$ we set
$$
\|y\|_{E^\times}:=\sup\{\sum\limits_{n=1}^\infty|x_ny_n|:\|x\|_{E}\leq
1\}.
$$
The space ($E^\times, \|\cdot\|_{E^\times}$) is a symmetric
sequence space (see \cite[Chapter II, \S 3]{KPS}).

We say that $\|\cdot\|_{E}$ is a Fatou-norm if given $x\in E$ and
a sequence $0\leq x_n\in E$ such that $x_n\:\uparrow\:x$, it
follows that $\|x_n\|_E\:\uparrow\:\|x\|_E$. It is known (see
\cite[Part I, Chapter X, \S 4, Theorem 7]{KA}) that
$\|\cdot\|_{E}$ is a Fatou-norm if and only if
$\|x\|_{E}=\|x\|_{E^{\times\times}}$ for every $x\in E$, where
$(E^{\times})^\times=E^{\times\times}$.

For a pair of sequences $x=(x_n)_{n=1}^\infty,
y=(y_n)_{n=1}^\infty\in l_\infty$ by $x\cdot y$ we denote the
sequence $(x_ny_n)_{n=1}^\infty$.

For any two symmetric sequence spaces $(E,\|\cdot\|_E)$ and
$(F,\|\cdot\|_F)$, we  set
\begin{equation}\label{e1}E^{F}:=\{x\in c_0 :
x\cdot y\in F, \;\text{for every}\; y\in E\}
\end{equation} and for $x\in
E^{F}$
\begin{equation}\label{e2}\|x\|_{E^{F}}:=\sup\limits_{\|y\|_E\leq
1}\|x\cdot y\|_F.\end{equation}

\begin{rem}
Another suggestive notation for the space $E^{F}$ introduced above
would be $F:E$  (see e.g. \cite{Hoff}). We use the notation
$E^{F}$ since it is in line with the notations from \cite{MP},
\cite{CDS} which are widely used in this section.
\end{rem}

The fact that the supremum in  \eqref{e2} above is finite for
every $x\in E^{F}$ is explained below.

\begin{rem}\label{r1} Any element $x\in E^{F}$ can be consider as
a linear bounded operator from $E$ into $F$ and $x\in E^{F}$  if
and only if for every $y\in F$ the following inequality holds
\begin{equation}\label{e6}\|x\cdot y\|_F\leq
\|x\|_{E^{F}}\|y\|_E.\end{equation} So the fact that the supremum
in Equation \eqref{e2} is finite can be proved via the closed
graph theorem considering the operator $x(y)=x\cdot y$ for every
$y\in E$.
\end{rem}

The proof of the following proposition is routine and is therefore
omitted.

\begin{prop}\cite[Theorem 4.4]{BCLS} $(E^{F},\|\cdot\|_{E^{F}})$ is a symmetric
sequence space.
\end{prop}

Analyzing definitions of the K\"{o}the dual and generalized K\"{o}the
dual spaces, it is not difficult to see that the spaces $E^{l_1}$
and $E^\times$ coincide (see also \cite{MP}).

In the following proposition we collect a number of known results
from \cite{MP}.

\begin{prop}\label{p11}

(i)\cite[p. 326, item (f)]{MP}  $l_\infty^{E}=E$;

(ii) \cite[Proposition 3]{MP}  if $1\leq r\leq p\leq\infty$ and
$\frac{1}{p}+\frac{1}{q}=\frac{1}{r}$, then $l_p^{l_r}=l_q$;

(iii) \cite[Theorem 2]{MP} if $1\leq p< r\leq\infty$, then
$l_p^{l_r}=l_\infty$.

\end{prop}

It is known (see e.g. \cite[Theorem 2]{MP}) that in the general
setting of Banach function spaces the space $E^{F}$ can be
trivial, that is $E^{F}=\{0\}$. The following proposition shows
that this is not the case in the setting of symmetric sequence
spaces.

\begin{prop}\label{p3}
$E^{F}\supset F$.
\end{prop}

\begin{proof}
Let $x\in l_\infty^{F}$ and $y\in E$, then obviously $y\in
l_\infty$ and from Remark \ref{r1} we have
$$
\|x\cdot y\|_F\leq \|x\|_{l_\infty^{F}}\|y\|_E.
$$
In particular $x\in E^{F}$ and
$\|x\|_{E^{F}}\leq\|x\|_{l_\infty^{F}}$. That is $E^{F}\supset
l_\infty^{F}$.

Since $F=l_\infty^{F}$ (see Proposition \ref{p11}, (i)), the claim
follows.
\end{proof}

The following proposition explains the connection between
generalized K\"{o}the duality and p-convexification.

\begin{prop}  \label{p6}\cite[Example 1]{MP}
$(E^p)^{l_p}=(E^\times)^p$.
\end{prop}

Let $E$, $F$ be sets of sequences, then we will denote $$E\cdot
F:=\{x=y\cdot z\:|\:y\in E, z\in F\}.$$

\begin{prop}\label{p777}\cite[Theorem 1]{JR}
$E\cdot E^\times=l_1$.
\end{prop}

When $E$ is $p$-convex, we can extend the result of Proposition \ref{p777} as follows.

\begin{prop}\label{p8}
If $E$ is $p$-convex, then $E\cdot(E)^{l_p}=l_p$.
\end{prop}

\begin{proof}
By the definition of the space $E^{l_p}$ we have that
$E\cdot(E)^{l_p}\subset l_p$.

We claim that $E\cdot(E)^{l_p}\supset l_p$.

Since $E$ is $p$-convex, the space $(E)^{1/p}$ is a symmetric
sequence space and we have that $((E)^{1/p})^p=E$ (see Remark
\ref{r2}). Denoting $Y=(E)^{1/p}$, we have $Y^p=E$ and
$(Y^\times)^p=E^{l_p}$ (see Proposition \ref{p6}). Thus our claim
is that $Y^p\cdot(Y^\times)^p\supset l_p$.

If $z\in l_p$, then  $|z|^p \in l_1$. By Proposition \ref{p777} we
have that exist $y_1\in Y$ and $y_2\in Y^\times$ such that
$|z|^p=y_1\cdot y_2$. Hence,
$|z|=(y_1y_2)^{1/p}=|y_1|^{1/p}|y_2|^{1/p}$, and
$z=\sgn(z)|y_1|^{1/p}|y_2|^{1/p}$.

Since $y_1\in Y$, we have $\sgn(z)|y_1|^{1/p}\in Y^p$. Similarly
we can obtain that $|y_2|^{1/p}\in (Y^\times)^p$. Thus the
inclusion $z\in Y^p\cdot(Y^\times)^p$ holds and the claim is
established.
\end{proof}

The second generalized K\"{o}the dual is defined by $E^{F
F}:=(E^{F})^{F}$.
 The following proposition is taken from
\cite{CDS} (see there Theorem 3.4 and Proposition 5.3,
respectively).

\begin{prop}\label{p9} (i)
If $E$ is $q$-concave for $1<q\leq \infty$, then $l_q^{E E}=l_q$.

(ii) If $E$ is $p$-convex for $1\leq p< \infty$, then
$E^{l_pl_p}=E$.

\end{prop}

\begin{rem}
  \label{r8} Since any symmetric sequence space $E$ is a solid subspace of $l_\infty$,
we have that $l_\infty\cdot E = E$.
\end{rem}

\section{Schur multipliers.}

Let $E$ and $F$ be symmetric sequence spaces. For every $A\in B(E, F)$, we set for brevity $$\|A\|_{E,
F}:=\|A\|_{B(E, F)}=\sup\limits_{\|x\|_E\leq 1}\|A(x)\|_F$$ and

 $$\|A\|_{1, F}:=\|A\|_{B(l_1, F)},\quad \|A\|_{E,
\infty}:=\|A\|_{B(E, l_\infty)}.$$

Any such operator $A$ can be identified with the matrix $A=
(a_{ij})_{i,j=1}^\infty$, whose every row represents an element from
$E^\times$ and every column represents an element from $F$. For a sequence
$x=(x_n)_{n\ge 1}\in E$, we have $A(x)=\{\sum\limits_{j}a_{ij}x_j\}_i\in
F.$

\begin{prop}\label{p12} If $E$ and $F$ are symmetric sequence spaces and $F$ has a Fatou-norm,
then $\|A\|_{E,F}=\|A^T\|_{F^\times,E^\times}$, where $A^T$ is the
transpose matrix for $A$.
\end{prop}

\begin{proof}
$$\|A\|_{E,F}=\sup\limits_{\|x\|_E\leq
1}\|A(x)\|_F=\sup\{\langle A(x),y\rangle:\|x\|_E\leq
1,\|y\|_{F^\times}\leq 1\}=$$ $$\sup\{\langle
x,A^T(y)\rangle:\|x\|_E\leq 1,\|y\|_{F^\times}\leq
1\}=\sup\limits_{\|y\|_{F^\times}\leq
1}\|A^T(y)\|_{E^\times}=\|A^T\|_{F^\times,E^\times}.$$
\end{proof}

The following preposition presents the  formulae for computing the
norm of  $A=(a_{ij})_{i,j=1}^\infty\in B(E, F)$ in some special
cases.

\begin{prop}\label{p19}\quad If $E$ is a symmetric sequence space,
then

(i) $\|A\|_{1,E}=\sup\limits_j\|(a_{ij})_i\|_E$;

(ii) $\|A\|_{E,\infty}=\sup\limits_i\|(a_{ij})_j\|_{E^\times}$.
\end{prop}

\begin{proof}

(i) By definition we have $$\|A\|_{1, E}=\sup\limits_{\|x\|_1\leq
1}\|A(x)\|_E=\sup\limits_{\|x\|_1\leq
1}\|\{\sum\limits_{j}a_{ij}x_j\}_i\|_E.$$ Hence, if $x=e_j$, where
$e_j=(e_j^k)_k\in c_0$ such, that $e_j^k=1$ for $j=k$ and
$e_j^k=0$ for $j\neq k$, then $\|A\|_{1, E}\geq\|\{a_{ij}\}_i\|_E$
for $j=1,2,...$, that is $\|A\|_{1,
E}\geq\sup\limits_j\|\{a_{ij}\}_i\|_E$.

Using the triangle inequality for the norm, we obtain the converse
inequality
$$\|A(x)\|_E=\|\{\sum\limits_{j}a_{ij}x_j\}_i\|_E\leq
\sum\limits_{j}|x_j|\|\{a_{ij}\}_i\|_E\leq$$ $$
\sup\limits_j\|\{a_{ij}\}_i\|_E\sum\limits_{j}|x_j|=
\sup\limits_j\|\{a_{ij}\}_i\|_E\|x\|_1.$$ Hence, $\|A\|_{1,
E}\leq\sup\limits_j\|\{a_{ij}\}_i\|_E$.

(ii) Applying Propositions \ref{p12} and \ref{p19} (i) we obtain
$\|A\|_{E,\infty}=\|A^T\|_{1,E^\times}=\sup\limits_i\|(a_{ij})_j\|_{E^\times}$.
\end{proof}

If  $A=(a_{ij})$ and $B=(b_{ij})$ are matrices of the same size
(finite or infinite), their \textit{Schur product} is defined to
be the matrix of elements-wise products $A\ast B=(a_{ij}b_{ij})$.

\begin{defn} An infinite matrix $M=(m_{ij})$ is called \textit{$(E,F)$-multiplier}
if $M\ast A\in B(E,F)$ for every  $A\in B(E,F)$.
\end{defn}

The set of all $(E,F)$-multipliers is denoted by
\begin{equation}\label{e5}\mathcal M(E,F):=\{M:M\ast A\in B(E,F),
\forall A\in B(E,F)\}.\end{equation} The collection $\mathcal
M(E,F)$ is a normed space with respect to the norm
\begin{equation}\label{e4}\|M\|_{(E,F)}:=\sup\limits_{\|A\|_{E,F}\leq 1}\|M\ast
A\|_{E,F}\end{equation} (when $E=l_p$ and $F=l_q$, we use the notation $\mathcal M(p,q)$ for \eqref{e5} and
$\|M\|_{(p,q)}$ for \eqref{e4}).

\begin{rem} \label{r3} (i) Viewing $M\in\mathcal M(E,F)$
as a linear operator $M : B(E, F) \to B(E, F)$, one easily checks that the supremum in \eqref{e4} is finite via the closed graph theorem.

(ii) Since $\|M\|_{(E,F )} = \sup\limits_{\|A\|_{E,F}\leq
1}\|M\ast A\|_{E,F}\geq \|M\ast u_{jk}\|_{E,F}=|m_{jk}|$ for every
$j, k = 1, 2, ...$, where $u_{jk} = (u^{(nm)}_{jk})_{nm}$ such
that $u^{(nm)}_{jk} = 1$ if $n = j$, $m = k$ and $u^{(nm)}_{jk}=
0$ otherwise, for $j, k, n,m = 1, 2, ...$, we have that
$\|M\|_{(E,F )}\geq\sup\limits_{j,k}|m_{jk}|$.
\end{rem}

The proofs of Theorem \ref{ThmBS} and Lemma \ref{l1} below are routine and incorporated here for
convenience of the reader.

\begin{thm}\label{ThmBS}

The normed space $(\mathcal M(E, F), \|\cdot\|_{(E,F)})$ is complete.

\end{thm}

\begin{proof}

 Since $ M(E, F) \subseteq B (B(E, F)), $ all we need to see is
  that~$M(E, F)$ is closed.  Take a sequence of~$M_n \in M(E, F)$, $n
  \geq 1$.  Let $$ M_n = \left( m^{(n)}_{jk} \right)_{j, k= 1}^\infty,\ \
  n \geq 1. $$ Assume that $ \lim_{n \rightarrow \infty} \left\| M_n
    - T \right\|_{B (B (E, F))} = 0, $ for some $ T \in B (B (E, F)). $

  Fix~$j, k \geq 0$ and fix a matrix unit~$u_{jk} \in B(E, F)$.  The
  sequence $ \left( M_n \right)_{n \geq 1} $ is Cauchy in~$B(B(E,
  F))$.  Consequently, the sequence $ \left( M_n \left(u_{jk}\right)
  \right)_{n \geq 1} $ is Cauchy in~$B(E, F)$.  Thus, the sequence $
  \left( m^{(n)}_{jk} \right)_{n \geq 0} $ is Cauchy in~$\mathbb R$.
  Thus, for every $j,k=1,2,...$ there is a number~$m_{jk}$ such that $$ \lim_{n \rightarrow
    \infty} \left| m^{(n)}_{jk} - m_{jk} \right| = 0. $$
    Since $$\left\|(M_n \left(u_{jk}\right) - m_{jk} u_{jk})(x)\right\|_F=
    \left\|(m_{jk}^{(n)} u_{jk} - m_{jk} u_{jk})(x)\right\|_F=$$ $$=\left|m_{jk}^{(n)} -
    m_{jk}\right| \; \left|x_k\right| \; \left\|e_k\right\|_E
    \leq\left|m_{jk}^{(n)} - m_{jk}\right| \; \left\|x\right\|_E$$ for every $x\in E$, we have
    $$\left\| M_n \left(e_{jk}\right) - m_{jk} e_{jk} \right\|_{E,F}\leq \left|m_{jk}^{(n)} - m_{jk}\right|.$$
    Hence
  \begin{equation}
    \label{limI}
    \lim_{n
      \rightarrow \infty} \left\| M_n \left(u_{jk}\right) - m_{jk} u_{jk} \right\|_{E,F} =
    0.
  \end{equation}

  On the other hand, the assumption $$ \lim_{n \rightarrow \infty}
  \left\| M_n - T \right\|_{B (B (E, F))} = 0 $$ implies that
  \begin{equation}
    \label{limII}
    \lim_{n \rightarrow
      \infty} \left\| M_n \left(u_{jk}\right) - T \left(u_{jk}\right) \right\|_{E,F} = 0.
  \end{equation}
  Combining~(\ref{limI}) with~(\ref{limII}) yields $$ T
  \left(u_{jk}\right) = m_{jk} u_{jk}\quad \text{for every}\;\; j,k=1,2,...\:. $$ That is, $T$ is a Schur
  multiplier.

\end{proof}

\begin{lem}\label{l1}
Let $A\in B(E, F)$ and $M=(m_{ij})$ be such that
$\sup\limits_{i,j}|m_{ij}|<\infty$. If $M\ast A$ maps $E$ into
$F$, then $M\ast A\in B(E, F)$.
\end{lem}

\begin{proof}
Since $A=(a_{ij})\in B(E, F)$, we have that
$$
\|(a_{ij})_j\|_{E^\times}<+\infty,\;\; \;\;\text{for all}\;\:
i=1,2,....
$$

Assume that $x=(x_i), x_n=(x_i^{(n)})\in E$, $y=(y_i)\in F$ are
such that $x_n\to x$ and  $(M\ast A)(x_n)\to y$. The claim of the
lemma follows from the closed graph theorem if we show that
$y=(M\ast A)(x)$. To that end, using Holder inequality, we have
that
$$
\left|\sum\limits_{j}m_{ij}a_{ij}\left(x_j^{(n)}-x_j\right)\right|\leq
\sum\limits_{j}|m_{ij}|\,|a_{ij}|\,|x_j^{(n)}-x_j|\leq$$ $$
\sup\limits_{i,j}|m_{ij}|
\sum\limits_{j}|a_{ij}|\,|x_j^{(n)}-x_j|\leq
\sup\limits_{i,j}|m_{ij}|\|x_n-x\|_E\|(a_{ij})_j)\|_{E^\times}\to
0,$$ for every $i=1,2,....$ It follows that

$$y_i=\lim\limits_{n\to\infty}\sum\limits_{j}m_{ij}a_{ij}x_j^{(n)}=\sum\limits_{j}m_{ij}a_{ij}x_j,\;\;
\;\;\text{for every}\;\; i=1,2,....
$$

So we obtain that $y=(M\ast A)(x)$.
\end{proof}

\begin{prop}\label{p23} If $E$ and $F$ are symmetric sequence spaces, $F$ has a Fatou-norm
and $M\in\mathcal M(E, F)$, then
$\|M\|_{(E,F)}=\|M^T\|_{(E^\times,F^\times)}.$
\end{prop}

\begin{proof}
Using Proposition \ref{p12}, we obtain that
$$\|M\|_{(E,F)}=\sup\limits_{\|A\|_{E,F}\leq 1}\|M\ast
A\|_{E,F}=\sup\limits_{\|A^T\|_{F^\times,E^\times}\leq 1}\|(M\ast
A)^T\|_{F^\times,E^\times}=$$
$$\sup\limits_{\|A^T\|_{F^\times,E^\times}\leq 1}\|M^T\ast
A^T\|_{F^\times,E^\times}=\sup\limits_{\|B\|_{F^\times,E^\times}\leq
1}\|M^T\ast
B\|_{F^\times,E^\times}=\|M^T\|_{(F^\times,E^\times)}.$$
\end{proof}

\begin{prop}\label{p22} For every symmetric sequence space $E$,
the following equations hold

(i) $\|M\|_{(1,E)}:=\|M\|_{(l_1,E)}= \sup\limits_{i,j}|m_{ij}|$;

(ii) $\|M\|_{(E,\infty)}:=
\|M\|_{(E,l_\infty)}=\sup\limits_{i,j}|m_{ij}|$.
\end{prop}

\begin{proof}

(i) As we have seen above $\|M\|_{(1,E)}\geq
\sup\limits_{i,j}|m_{ij}|$ (Remark \ref{r3}). Let us prove the
converse inequality. For every operator $A=(a_{ij})\in B(E,F)$,
using Proposition \ref{p19} (i), we have $$\|M\ast
A\|_{1,E}=\sup\limits_j\|(m_{ij}a_{ij})_i\|_E\leq
\sup\limits_j(\sup\limits_i|m_{ij}|\|(a_{ij})_i\|_E)$$ $$\leq
\sup\limits_j\sup\limits_i|m_{ij}|\sup\limits_j\|(a_{ij})_i\|_E=
\sup\limits_{i,j}|m_{ij}|\|A\|_{1,E}.$$ Hence, $\|M\|_{(1,E)}\leq
\sup\limits_{i,j}|m_{ij}|$.

(ii) The claim follows from Proposition \ref{p23} and (i) above.
\end{proof}

\begin{cor}\label{th88}
The multiplier spaces  $\mathcal M(1, E)$ and $\mathcal
M(F,\infty)$ are isometrically isomorphic and do not depend on a
choice of the spaces $E$ and $F$.
\end{cor}

Now, we are well equipped to consider the question of embedding of the
$(E,F)$-multiplier space into an $(p,q)$-multiplier space. We
start by recalling the following result from \cite{Ben}.

Analyzing the proof of \cite[Theorem
6.1]{Ben}, we restate its result as follows.

\begin{thm}\label{th4}(i) If $1 \leq p_2\leq p_1\leq\infty$ and $1\leq q_1\leq
q_2\leq\infty$, then $\mathcal M(p_1,q_1)\subseteq \mathcal
M(p_2,q_2)$;

(ii) If $1 \leq q_1,q_2\leq\infty$ and $p_1=p_2=1$, then $\mathcal
M(p_1,q_1)= \mathcal M(p_2,q_2)$;

(iii) If $1 \leq p_1,p_2\leq\infty$ and $q_1=q_2=\infty$, then
$\mathcal M(p_1,q_1)= \mathcal M(p_2,q_2)$;

(iv) If $1 \leq p_1,q_1\leq\infty$ and $q_2\leq 2\leq p_2$, then
$\mathcal M(p_2,q_2)\subseteq \mathcal M(p_1,q_1)$.

\end{thm}

The following corollary follows immediately from Theorem \ref{th4}
(i) and (iv).

\begin{cor}\label{corollary}
$\mathcal M(2,2)=\mathcal M(\infty,1)$.
\end{cor}

The following theorem extends Theorem \ref{th4} and is the main result of this section.

\begin{thm}\label{th5}
Let $1\leq p, q\leq \infty$ be given. If $E$ is $p$-convex and $F$
is $q$-concave, then $\mathcal M(E,F)\subseteq\mathcal M(p,q)$.
\end{thm}

\begin{proof}
Let $A=(a_{ij})_{i,j=1}^\infty\in B(l_p, l_q)$ and $M\in\mathcal
M(E,F)$. To prove the claim of the Theorem \ref{th5}, we need to
show that $M\ast A\in B(l_p, l_q)$.

For $x\neq 0 \in E^{l_p}$, $y\neq 0 \in l_q^{F}$ consider the
operator $yAx$ which acts on the element $z\in E$ as follows:
$$
(yAx)(z)=y\cdot(A(x\cdot z))=(\sum\limits_n y_ma_{mn}x_nz_n)_m.
$$
Since $z\in E$, $x\in E^{l_p}$, by the definition of the space
$E^{l_p}$, we have that $x\cdot z\in l_p$. By the assumption $A\in
B(l_p, l_q)$, we have  $A(x\cdot z)\in l_q$. Since $y\in l_q^{F}$,
we obtain $y\cdot(A(x\cdot z))\in F$. That is, the operator $yAx$
maps the space $E$ into the space $F$. Moreover, for $z\in E$,
using \eqref{e6}, we have  $$ \|(yAx)(z)\|_F=\|y\cdot(A(x\cdot
z))\|_F\leq\|y\|_{l_q^{F}}\|A(x\cdot z)\|_q\leq$$
$$\|y\|_{l_q^{F}}\|A\|_{p,q}\|x\cdot
z\|_p\leq\|y\|_{l_q^{F}}\|A\|_{p,q}\|x\|_{E^{l_p}}\|z\|_E. $$
Hence, $yAx\in B(E, F)$. By the assumption we have $M\in\mathcal
M(E,F)$, and so we obtain $M\ast(yAx)\in B(E, F)$. It is easy to
see that $M\ast(yAx)=y(M\ast A)x$. Therefore $y(M\ast A)x\in B(E,
F)$.

The next step is to prove that the operator $M\ast A$ maps $l_p$
into $l_q$. Let $x\in l_p$ be given. Since $F$ is $q$-concave, by
Proposition \ref{p9} (i) we have that if $y\cdot z\in F$ for every
$y\in l_q^{F}$, then $z\in l_q$. We claim that $y\cdot((M\ast
A)(x))\in F$ for every $y\in l_q^{F}$. Since $E$ is $p$-convex and
$x\in l_p$, there exist $x_1\in E^{l_p}$ and $x_2\in E$ such that
$x=x_1\cdot x_2$ (see Proposition \ref{p8}). Since $y(M\ast A)x\in
B(E, F)$ for every $x\in E^{l_p}$ and $y\in l_q^{F}$, we have
$$
y\cdot((M\ast A)(x))=y\cdot((M\ast A)(x_1\cdot x_2))=(y(M\ast
A)x_1)(x_2)\in F,
$$
We conclude $(M\ast A)(x)\in l_q$, that is, $M\ast
A:l_p\rightarrow l_q$. Since $M\in\mathcal M(E,F)$, we have that
$\sup\limits_{i,j}|m_{ij}|<\infty$ (Remark \ref{r3}, (ii)). By
Lemma \ref{l1} we obtain that $M\in \mathcal M(p,q)$ (see also
\cite[Lemma 2]{Sar}).
\end{proof}

Since any symmetric sequence space is 1-convex and
$\infty$-concave (Remark \ref{r4}), the following corollary
follows immediately from Theorem \ref{th5}.

\begin{cor}\label{c1}
For every pair of symmetric sequence spaces $E$ and $F$, we have
$$\mathcal M(E,F)\subseteq\mathcal M(1,\infty).$$
\end{cor}

\begin{cor}\label{c2}
If $E$ is $p$-convex, $F$ is $q$-concave, then $\mathcal
M(E,F)\subseteq\mathcal M(p_1,q_1)$ for every $1\leq p_1\leq
p\leq\infty$ and $1\leq q\leq q_1\leq\infty$.
\end{cor}

\begin{proof}
By the assumptions of the Corollary and by Theorem \ref{th5} we
have that $\mathcal M(E,F)\subseteq\mathcal M(p,q)$. From Theorem
\ref{th4} (i) it follows that $\mathcal M(p,q)\subseteq\mathcal
M(p_1,q_1)$. Hence, the claim follows.
\end{proof}

The proof of the following theorem is similar to the proof of
Theorem \ref{th5}. We need to apply Proposition \ref{p9} (ii) with
$p=1$ and Remark \ref{r8} instead of Propositions \ref{p9} (i) and
\ref{p8}, respectively. We supply details for completeness.

\begin{thm}\label{p10}
For symmetric sequence spaces $E$ and $F$ the embedding $\mathcal
M(\infty,1)\subseteq\mathcal M(E,F)$ holds.
\end{thm}

\begin{proof}
Let $A=(a_{ij})_{i,j=1}^\infty\in B(E,F)$ and $M\in\mathcal
M(\infty,1)$. To prove the claim of the Theorem \ref{th5}, we need
to show that $M\ast A\in B(E,F)$.

For $x\neq 0 \in l_\infty^E$, $y\neq 0 \in {F}^{l_1}$ consider the
operator $yAx$ which acts on the element $z\in l_\infty$ as
follows:
$$
(yAx)(z)=y\cdot(A(x\cdot z))=(\sum\limits_n y_ma_{mn}x_nz_n)_m.
$$
Since $z\in l_\infty$, $x\in l_\infty^E$, by the definition of the
space $l_\infty^E$, we have that $x\cdot z\in E$. By the
assumption $A\in B(E,F)$, we have  $A(x\cdot z)\in F$. Since $y\in
{F}^{l_1}$, we obtain $y\cdot(A(x\cdot z))\in l_\infty$. That is,
the operator $yAx$ maps the space $l_\infty$ into the space $l_1$.
Moreover, for $z\in l_\infty$, using \eqref{e6}, we have  $$
\|(yAx)(z)\|_1=\|y\cdot(A(x\cdot
z))\|_1\leq\|y\|_{{F}^{l_1}}\|A(x\cdot z)\|_F\leq$$
$$\|y\|_{{F}^{l_1}}\|A\|_{E,F}\|x\cdot
z\|_E\leq\|y\|_{{F}^{l_1}}\|A\|_{E,F}\|x\|_{{l_\infty}^E}\|z\|_\infty.
$$ Hence, $yAx\in B(l_\infty, l_1)$. By the assumption we have
$M\in\mathcal M(\infty,1)$, and so we obtain $M\ast(yAx)\in
B(l_\infty, l_1)$. It is easy to see that $M\ast(yAx)=y(M\ast
A)x$. Therefore $y(M\ast A)x\in B(l_\infty, l_1)$.

The next step is to prove that the operator $M\ast A$ maps $E$
into $F$. Let $x\in E$ be given. By Proposition \ref{p9} (ii) with
$p=1$ we have that if $y\cdot z\in l_1$ for every $y\in F^{l_1}$,
then $z\in F$. We claim that $y\cdot((M\ast A)(x))\in l_1$ for
every $y\in F^{l_1}$. Since  $x\in E$, there exist $x_1\in
l_\infty^F$ and $x_2\in l_\infty$ such that $x=x_1\cdot x_2$ (see
Remark \ref{r8}). Since $y(M\ast A)x\in B(l_\infty, l_1)$ for
every $x\in l_\infty^F$ and $y\in F^{l_1}$, we have
$$
y\cdot((M\ast A)(x))=y\cdot((M\ast A)(x_1\cdot x_2))=(y(M\ast
A)x_1)(x_2)\in l_1,
$$
We conclude $(M\ast A)(x)\in l_q$, that is, $M\ast A:E\rightarrow
F$. Since $M\in\mathcal M(\infty,1)$, we have that
$\sup\limits_{i,j}|m_{ij}|<\infty$ (Remark \ref{r3}, (ii)). By
Lemma \ref{l1} we obtain that $M\in \mathcal M(E,F)$.
\end{proof}

Note that Corollary \ref{c1} and Theorem \ref{p10} imply the existence of the maximal and minimal multiplier spaces, that is we have
$$\mathcal M(\infty,1)\subseteq\mathcal M(E,F)\subseteq\mathcal
M(1,\infty).$$

The following theorem gives sufficient conditions on $E$ and $F$ guaranteeing the equality $\mathcal M(\infty,1)=\mathcal M(E,F)$.

\begin{thm}\label{th6}
If $E$ is $2$-convex, $F$ is $2$-concave and $1\leq q\leq 2\leq
p\leq\infty$, then
$$\mathcal M(E,F)=\mathcal M(p,q).$$
\end{thm}

\begin{proof}

It is sufficient to prove the assertion for the case $p=q=2$.
Indeed, for $q\leq 2\leq p$  embeddings $\mathcal
M(p,q)\subseteq\mathcal M(2,2)$ and $\mathcal
M(2,2)\subseteq\mathcal M(p,q)$  follow then from Theorem
\ref{th4} items (i) and (iv), respectively.

By Theorem \ref{th5} we have $\mathcal M(E,F)\subseteq\mathcal
M(2,2)$. Using Theorem \ref{th4} (iv) we obtain $\mathcal
M(2,2)\subseteq \mathcal M(\infty,1)$. Theorem \ref{p10} yields the converse embedding. The proof is completed.
\end{proof}

The following corollaries follow immediately from \cite[Theorem
5.1 (i)]{Pis} and Theorem \ref{th6}.

\begin{cor}\label{c3}
Let $E$ be $2$-convex and $F$ be $2$-concave. A matrix
$M=(m_{ij})_{i,j=1}^\infty$ is an element of the space $\mathcal
M(E,F)$ if and only if there is a Hilbert space $H$ and families
$(y_i)_{i=1}^\infty$, $(x_j)_{j=1}^\infty$ of elements of $H$ such
that $m_{ij}=\langle y_i, x_j\rangle$, for every $(i,j)\in \mathbb
N\times \mathbb N$ and
$\sup\limits_i\|y_i\|\sup\limits_j\|x_j\|<\infty$.
\end{cor}

By $E\otimes F$ we shall denote the algebraic tensor product of
$E$ and $F$. We introduce the tensor norm $\gamma_2^*$ defined as
follows. For all $u$ in $E\otimes F$ we define
$$
\gamma_2^*(u):=\inf\{(\sum\limits_j\|x_j\|_E^2)^{1/2}(\sum\limits_i\|\xi_i\|_F^2)^{1/2}\},
$$
where the infimum runs over all finite sequences $(x_j)_{j=1}^n$
in $E$ and $(\xi_i)_{i=1}^n$ in $F$ such that
 $u=\sum\limits_{i=1}^n x_i\otimes \xi_i$. It is not difficult to check that $\gamma_2^*$ is a norm on
$E\otimes F$. We will denote by $E\widehat{\otimes}_{\gamma_2^*}
F$ the completion of $E\otimes F$ with respect to that norm.

The next result follows from Theorem
\ref{th6} and \cite[Theorem 5.1 (ii) and Theorem 5.3]{Pis}.

\begin{cor}\label{c3}
Let $E$ be $2$-convex and $F$ be $2$-concave. Then $$\mathcal
M(E,F)=(l_1\widehat{\otimes}_{\gamma_2^*} l_1)^*.$$
\end{cor}

We complete this section with the following observation.

\begin{prop}\label{r5}
The embeddings in Theorems and Corollaries \ref{th4}- \ref{th6}
are continuous.
\end{prop}

\begin{proof}
For example, we will prove the claim for
Theorem \ref{th4} (i). Let $p_1, q_1$ be the same as in Theorem
\ref{th4} (i) and $I:\mathcal M(p_1,q_1)\to \mathcal M(p_2,q_2)$
is an operator of embedding, that is $I(M)=M$ for every $M\in
\mathcal M(p_1,q_1).$  Let the sequence $(M_n)_{n\geq 1}\subset
\mathcal M(p_1,q_1)$ be such that $M_n\to 0$ in the space
$\mathcal M(p_1,q_1)$ for $n\to\infty$ and $I(M_n)=M_n\to M$ in
the space $\mathcal M(p_2,q_2)$ for $n\to\infty$. Uisng the notations from the proof of Proposition \ref{p19}, we have
$$\langle M_n(u_{jk}), e_k\rangle= m^{(n)}_{jk}e_j$$ and
$$\langle M(u_{jk}), e_k\rangle= m_{jk}e_j$$
for every $j,k,n=1,2,...\:.$ Since $M_n\to 0$ in the space
$\mathcal M(p_1,q_1)$, we obtain that $m^{(n)}_{jk}\to 0$ for
$n\to \infty$ and $j,k=1,2,...\:.$ By the assumption $M_n\to M$ in
the space $\mathcal M(p_2,q_2)$ we have that $m^{(n)}_{jk}\to
m_{jk}$. Hence, we obtain $m_{jk}=0$ for all
$j,k=1,2,...\:,$ that is $M=0$. Thanks to Theorem \ref{ThmBS} we may apply the closed graph theorem and conclude that the operator $I$ is bounded.
\end{proof}

\section{The main triangle projector}

As before,  the $n$-th main triangle projection is denoted by $T_n$ ($n\in \mathbb N$). The question when the sequence
$(\|T_n\|_{B(l_p,l_q)\to B(l_p,l_q)})_{n\geq 1}$ is (un)bounded was
completely answered in \cite{Ben1} and \cite{Pil}.

\begin{prop}\label{p18}\cite[Proposition 1.2]{Pil}
Let $p\neq 1$, $q\neq\infty$ и $q\leq p$. Then we have
$$
\|T_n\|_{(p,q)}\geq C(p,q) \ln n,
$$
where $C(p,q)$ is a constant dependent only on $p$ and $q$.
\end{prop}

The following proposition extends the result of Proposition
\ref{p18} to a wider class of symmetric sequence spaces.

\begin{prop}\label{p111}
Let $E$ and $F$ be symmetric sequence spaces. If $E$ is $p$-convex
and $F$ is $q$-concave for $p\neq 1$, $q\neq\infty$ and $q\leq p$,
then
$$
\|T_n\|_{(E,F)}\geq C(p,q) \ln n,
$$
where $C(p,q)$ is a constant dependent only on $p$ and $q$.
\end{prop}

\begin{proof}
Consider the space of multipliers $\mathcal M(E,F)$. Since the
operator $T_n$ is a finite rank operator, we have that
$T_n\in\mathcal M(E,F)$ and also $T_n\in\mathcal M(p,q)$. Since
$E$ is $p$-convex and $F$ is $q$-concave, by Theorem \ref{th5} we
have that $\mathcal M(E,F)\subset\mathcal M(p,q)$. By Proposition \ref{r5}, there exists a constant $C_1$ such that
$$\|T_n\|_{(E,F)}\geq C_1\|T_n\|_{(p,q)},\quad \forall n\ge 1.$$ Applying
Preposition \ref{p18}, we conclude
$$\|T_n\|_{(E,F)}\geq C_1\|T_n\|_{(p,q)}\geq
C_1C(p,q) \ln n,\quad \forall n\ge 1.$$
\end{proof}

\section{Projective and injective tensor products of symmetric sequence spaces}

We briefly recall some notions and notations from \cite{Pil}.

Let $\mathcal M_0$ be the set of scalar-valued (real or complex)
infinite matrices, such that if $A=(a_{ij})\in \mathcal M_0$, then
$a_{ij}\neq 0$ for all but finitely many $(i,j)\in \mathbb N\times
\mathbb N$.

A non-negative function $\|\cdot\|_{\mathcal M_0}$ on $\mathcal
M_0$ is called a \textit{matrix norm}, if it satisfies the
following conditions:

(i)  for every $A,B\in \mathcal M_0$ and for any scalar $\alpha$

$\|A\|_{\mathcal M_0}=0$ iff $A=0$;

$\|\alpha A\|_{\mathcal M_0}=|\alpha|\|A\|_{\mathcal M_0}$;

$\|A+B\|_{\mathcal M_0}\leq \|A\|_{\mathcal M_0}+\|B\|_{\mathcal
M_0}$;

(ii) $\|u_{jk}\|_{\mathcal M_0}=1$, $\forall j,k\ge 1$ (see the definition of the matrix unit $u_{jk}$ in Remark \ref{r3}(ii)).

(iii) $\|P_{nm}(A)\|_{\mathcal M_0}\leq \|A\|_{\mathcal M_0}$ for
all $A\in \mathcal M_0$, $n,m=1,2,...$, where $P_{nm}$ is
projector on the first $n$ lines and $m$ columns.

A matrix norm is called \textit{unconditional} if

(iv) $\|A\|_{\mathcal M_0}=\|(x_{ij}a_{ij})_{ij}\|_{\mathcal
M_0}$, for all $A\in \mathcal M_0$, where $x_{ij}=\pm 1$,
$i,j=1,2,...$.

An unconditional matrix norm is called \textit{symmetric} if

(v) $\|A\|_{\mathcal
M_0}=\|(a_{\varphi(i)\psi(j)})_{ij}\|_{\mathcal M_0}$ for all
$A\in \mathcal M_0$ and for all permutations $\varphi$, $\psi$ of
positive integers.

If $\|\cdot\|_{\mathcal M_0}$ is a matrix norm, then the conjugate
norm defined by
$$
\|A\|_{\mathcal M_0}^*:=\sup\{|\sum\limits_{i,j}a_{ij}b_{ij}|:
B\in \mathcal M_0, \|B\|_{\mathcal M_0}\leq 1\}
$$
We have $\|A\|_{\mathcal M_0}^{**}=\|A\|_{\mathcal M_0}$.

We denote
$$\|T_n\|_{(\mathcal M_0)}:=\sup\{\|T_n(A)\|_{\mathcal
M_0}:\|A\|_{\mathcal M_0}\leq 1\}.$$

It is known that (see \cite[Equation (1.1)]{Pil})
\begin{equation}\label{e7}\|T_n\|_{(\mathcal
M_0)}^*:=\sup\{\|T_n(A)\|^*: A\in \mathcal M_0, \|A\|^*\leq
1\}=\|T_n\|_{(\mathcal M_0)}.\end{equation}

The following theorem reflects the connection between the
boundedness of the norms of the main triangle projections and the
question concerning a possibility of embedding of a matrix space
into a Banach space with an unconditional basis.

\begin{thm}\label{th8}\cite[Theorem 2.3]{Pil} Let $\|\cdot\|_{\mathcal M_0}$
be a such symmetric matrix norm. If the sequence
$\{\|T_n\|_{(\mathcal M_0)}\}_n$ is unbounded, then the space
$(\mathcal M_0, \|\cdot\|_{\mathcal M_0})$ is not isomorphic to
any subspace of a Banach space with an unconditional basis.
\end{thm}

In the present paper, we consider only two types of matrix spaces,
projective and injective tensor products. Recall the definition of
the spaces (see e. g. \cite{Ryan}).

Let $(E,\|\cdot\|_E)$, $(F,\|\cdot\|_F)$ be Banach spaces over the
field $\mathbb K$ (real or complex numbers). By $E\otimes F$ we denote the algebraic tensor product of $E$
and $F$.

For every $u\in E\otimes F$  we define the \textit{projective
tensor norm}
$$\pi(u):=\inf\{\sum\limits_{i=1}^n \|x_i\|_E\|y_i\|_F:
u=\sum\limits_{i=1}^n x_i\otimes y_i\}$$ (respectively,
\textit{injective tensor norm}
$$\varepsilon(u):=\sup\{|\sum\limits_{i=1}^n
\varphi(x_i)\psi(y_i)|: \; \varphi\in E^*, \|\varphi\|_{E^*}\leq
1, \psi\in F^*, \|\psi\|_{F^*}\leq 1\}).$$

The completion of $E\otimes F$ with respect to the norm $\pi$
(respectively, $\varepsilon$) we shall denote by
$E\widehat{\otimes} F$
 (respectively, $E\widehat{\widehat{\otimes}} F$) and call by \textit
{projective} (respectively, \textit{injective}) \textit{tensor
product of Banach spaces $E$ and $F$}.

For convenience, we denote the norm  $\pi$ (respectively,
$\varepsilon$) on the space $E\otimes F$ by $\pi_{E,F}$
(respectively, $\varepsilon_{E,F}$).


Let $c_{00}$ be the linear space of all finitely supported
sequences. The tensor product $c_{00}\otimes c_{00}$ can be
identified with the space of matrices $\mathcal M_0$ on $\mathbb
K$. The tensor product basis $\{e_j\otimes e_k\}_{j,k=1}^\infty$
corresponds to the standard basis in $\mathcal M_0$ (see
\cite[\S1.5]{Ryan} and \cite[\S3]{Pil}).

If $E$ is separable and $p$-convex, $F$ is $q$-concave symmetric
sequence spaces, then the spaces $E^*$ and $F^*$ are symmetric
spaces too (see Remark \ref{r6}), furthermore their dual spaces coincide
with $E^\times$ and $F^\times$, respectively (see \cite[Part I,
Chapter X, \S 4, Theorem 1]{KA}). Therefore $\varepsilon_{E,F}$
and $\pi_{E,F}$ are symmetric matrix norm on the space
$c_{00}\otimes c_{00}$. For this reason, below we shall only
consider separable symmetric sequence spaces.

The following proposition explains the connection between tensor
product norms and the operator norm in $B(E,F)$ (see \cite[\S2.2 and 3.1]{Ryan}).

\begin{prop}\label{p15}
(i) The norm $\varepsilon_{E,F}$ coincides with the operator norm
on the space $B(E^\times,F)$;

(ii) the conjugate norm  for the norm $\pi_{E,F}$ coincides with
the operator norm on the space $B(E,F^\times)$.
\end{prop}

\begin{rem} In particular, Proposition \ref{p15} (i) shows that for $A=(a_{ij})\in B(E^\times,F)$, we have
$$\|A\|_{\varepsilon_{E,F}}=\sup\{\left|\sum\limits_{i,j}a_{ij}x_iy_j\right|:
\|x\|_{E^\times}\leq 1, \|y\|_{F^\times}\leq 1\}.$$
Another important observation is

\begin{equation}\label{e8}\|T_n\|_{(\pi_{E,F})}=\|T_n\|_{(\varepsilon_{F^\times,E^\times})}^*\qquad\text{for every}\quad n\geq 1.\end{equation}

\end{rem}

We can reformulate Proposition \ref{p111} as follows:

\begin{prop}\label{p4}
Let $E$ and $F$ be symmetric sequence spaces.

(i) if $E$ is $p$-concave and $F$ is $q$-concave for $p\neq
\infty$, $q\neq\infty$ and $q\leq p^*$, then
$$
\|T_n\|_{(\varepsilon_{E,F})}\geq C(p,q) \ln n.
$$

(ii)  if $E$ is $p$-convex and $F$ is $q$-convex for $p\neq 1$,
$q\neq 1$ and $p^*\leq q$ then
$$
\|T_n\|_{(\pi_{E,F})}\geq C(p,q) \ln n.
$$
\end{prop}

\begin{proof}
(i)  Since the norm $\varepsilon_{E,F}(\cdot)$ coincides with the
norm $\|\cdot\|_{E^\times,F}$ (see Proposition \ref{p15} (i)) and
$E^\times$ is $p^*$-convex (see Proposition \ref{p2}), by
Preposition \ref{p111}  we have that
$$\|T_n\|_{(\varepsilon_{E,F})}=\|T_n\|_{(E^\times,F)}\geq
C(p,q) \ln n.$$

(ii) Applying \eqref{e8} and \eqref{e7}, we have
\begin{equation}\label{e9}
\|T_n\|_{(\pi_{E,F})}=\|T_n\|_{(\varepsilon_{F^\times,E^\times})}^*=\|T_n\|_{(\varepsilon_{F^\times,E^\times})}.
\end{equation}
Since $E$ (respectively, $F$) is $p$-convex  (respectively, $q$-convex), we
have $E^\times$ (respectively, $F^\times$) is $p^*$-concave (respectively, $q^*$-concave)
(see Proposition \ref{p2}). By the item (i) above, we have
\begin{equation}\label{e10}
\|T_n\|_{(\varepsilon_{F^\times,E^\times})}\geq C(p,q) \ln n,\quad \forall n\ge 1
\end{equation}
whenever $p^*\neq \infty$, $q^*\neq\infty$ and $q^*\leq p$. Applying \eqref{e9}  and  \eqref{e10}, we obtain that
$$
\|T_n\|_{(\pi_{E,F})}=\|T_n\|_{(\varepsilon_{F^\times,E^\times})}\geq
C(p,q) \ln n ,\quad \forall n\ge 1
$$
for $p\neq 1$, $q\neq 1$ and $p^*\leq q$.
\end{proof}

The following theorem is the main result of the section.

\begin{thm}\label{th9}
Let $E$ and $F$ be symmetric sequence spaces.

(i) if $E$ is $p$-concave and $F$ is $q$-concave for $p\neq
\infty$, $q\neq\infty$ and $q\leq p^*$, then the tensor product
$E\widehat{\widehat{\otimes}} F$ is not isomorphic to any subspace
of a Banach space with an unconditional basis.

(ii) if $E$ is  $p$-convex and $F$ is $q$-convex for $p\neq 1$,
$q\neq 1$ и $p^*\leq q$, then the tensor product
$E\widehat{{\otimes}} F$ is not isomorphic to any subspace of a
Banach space with an unconditional basis.
\end{thm}

\begin{proof}
(i) By Proposition \ref{p4} (i), with $p\neq \infty$,
$q\neq\infty$ and $q\leq p^*$,  we have that the sequence
$\{\|T_n\|_{(\varepsilon_{E,F})}\}_n$ is unbounded. By Theorem
\ref{th8} we obtain that the space $c_{00}\otimes c_{00}$ with the
norm $\varepsilon_{E,F}$ is not isomorphic to any subspace of a
Banach space with an unconditional basis. Since $(c_{00}\otimes
c_{00}, \varepsilon_{E,F})$ is a linear subspace in
$E\widehat{\widehat{\otimes}} F$, we conclude that the space
$E\widehat{\widehat{\otimes}} F$ is not isomorphic to any subspace
of a Banach space with an unconditional basis.

(ii) Similarly to the item (i), using Proposition \ref{p4} (ii)
instead of Proposition \ref{p4} (i).
\end{proof}

Now we consider a class of symmetric sequence spaces,
called Orlicz-Lorentz sequence spaces, generalizing the class of $l_p$-spaces. For detailed studies of
this class of spaces we refer to \cite{Hud},\cite{KA1} and
\cite{KA3}.

We recall that $G:[0,\infty)\to[0,\infty)$ is an \textit{Orlicz
function}, that is, a convex function which assumes value zero
only at zero) and $w=(w_k)$ is a \textit{weight sequence}, a
non-increasing sequence of positive reals such that
$\sum\limits_{k=1}^\infty w_k=\infty$.

The \textit{Orlicz-Lorentz sequence space} $\lambda_{w,G}$ is
defined by
$$
\lambda_{w,G}:=\{x=(x_k):\sum\limits_{k=1}^\infty G(\lambda
x_k^*)w_k<\infty\;\;\text{for some}\:\: \lambda>0\}.
$$
It is easy to check that $\lambda_{G,w}$ is a symmetric sequence
space, equipped with the norm
$$\|x\|_{w,G}:=\inf\{\lambda>0:\sum\limits_{k=1}^\infty G(\lambda
x_k^*)w_k\leq 1\}.$$ Two Orlicz functions $G_1$ and $G_2$ are said
to be \textit{{equivalent}} if there exist such a constant
$c<\infty$ that
$$
G_1(c^{-1}t)\leq G_2(t)\leq G_1(ct),\quad \text{for every}\quad
t\in [0,\infty).
$$

The following theorem indicates sufficient conditions under which
the space $\lambda_{w,G}$ is $p$-convex or $q$-concave (see also
\cite{Kam}).

\begin{thm}\label{th2}\cite[Theorem 5.1]{SMS}\quad
Let $G$ be an Orlicz function, $w=(w_k)$ be a weight sequence and
$1<p,q<\infty$. Then the following claims hold:

(i) If $G\circ t^{1/p}$ is equivalent to a convex function and
$\sum\limits_{k=1}^n w_k$ is concave, then the space
$\lambda_{w,G}$ is $p$-convex;

(ii) If $G\circ t^{1/q}$ is equivalent to a concave function and
$\sum\limits_{k=1}^n w_k$ is convex, then the space
$\lambda_{w,G}$ is $q$-concave.

\end{thm}

According to Theorem \ref{th2}, we can reformulate Theorem
\ref{th9} for Orlicz-Lorentz sequence spaces as follows.

\begin{thm}\label{th3} Let $G_1$ and $G_2$ be Orlicz functions
and $w_1=(w_k^{(1)})$, $w_2=(w_k^{(2)})$ be weight sequences such
that the spaces $\lambda_{w_1,G_1}$ and $\lambda_{w_2,G_2}$ are
separable.
 If  $G_1\circ t^{1/p}$ and $G_2\circ t^{1/q}$ are equivalent
to concave (convex, respectively) functions for $p\neq \infty$,
$q\neq\infty$ and $q\leq p^*$ ($p\neq 1$, $q\neq1$ and $p^*\leq
q$, respectively),  and $\sum\limits_{k=1}^n w_k^{(1)}$,
$\sum\limits_{k=1}^n w_k^{(2)}$ are convex (concave, respectively)
functions, then the tensor product
$\lambda_{w_1,G_1}\widehat{\widehat{\otimes}} \lambda_{w_2,G_2}$
($\lambda_{w_1,G_1}\widehat{{\otimes}} \lambda_{w_2,G_2}$,
respectively) is not isomorphic to any subspace of a Banach space
with an unconditional basis.
\end{thm}

\end{document}